\documentclass[a4paper, 12pt]{article}

\usepackage[margin=2cm]{geometry}
\usepackage{hyperref,xcolor}
\usepackage{amsthm, amssymb, amsmath}
\usepackage{enumerate}
\usepackage{algorithmic}
\usepackage{authblk}
\usepackage{graphics}
\usepackage{pstricks}

\hypersetup{
    colorlinks=false,
	linkcolor=red,
	citecolor=red,
    pdfborder={0 0 0},
}

\title{A combinatorial proof of Fisher's Inequality}

\author[1]{Rogers~Mathew\footnote{This author was supported by a grant from the Science and Engineering Research
Board, Department of Science and Technology, Govt. of India (project number: MTR/2019/000550).}}
\author[2]{Tapas Kumar Mishra}

\affil[1]
{
	Department of Computer Science and Engineering, \authorcr
	Indian Institute of Technology, Hyderabad \authorcr
	rogers@iith.ac.in
}
\affil[2]
{
	Department of Computer Science and Engineering, \authorcr
	National Institute of Technology, Rourkela \authorcr
 mishrat@nitrkl.ac.in
}

\theoremstyle{definition}

\theoremstyle{plain}
\newtheorem{theorem}{Theorem}

\theoremstyle{remark}

\newtheoremstyle{plainitshape}
  {}
  {}
  {\itshape}
  {}
  {\itshape}
  {.}
  {0.5em}
  {}
\theoremstyle{plainitshape}

\newtheoremstyle{cases}
  {}
  {}
  {}
  {}
  {}
  {\newline}
  {0.5em}
  {{\itshape \thmname{#1}} \thmnumber{#2} ({\itshape\thmnote{#3}}).\medskip}
\theoremstyle{cases}

\newtheoremstyle{constructions}
  {}
  {}
  {}
  {}
  {}
  {}
  {0.5em}
  {{\itshape \thmname{#1}} \thmnumber{#2}\medskip}
\theoremstyle{constructions}

\definecolor{ao(english)}{rgb}{0.0, 0.5, 0.0}

\date{}

\begin{document}
\maketitle
\begin{abstract}
In this note, we give a simple, counting based proof of Fisher's Inequality that does not use any tools from linear algebra.  
\end{abstract}
\section{Introduction}
Let $k$ be a positive integer and let $\mathcal{A}$ be a family of subsets of $[n]$. Fisher's Inequality states that if the cardinality of the intersection of every pair of distinct sets in $\mathcal{A}$ is $k$, then $|\mathcal{A}| \leq n$. R. A. Fisher \cite{fisher1940examination} while studying Balanced Incomplete Block Designs (BIBDs) proved that the number of points never exceeds the number of blocks. R.C. Bose \cite{bose1949} proved the Fisher's inequality when all the sets in the family $\mathcal{A}$ are of the same size. In \cite{de1948combinatorial}, it was shown that a maximal family of subsets of $[n]$ that has exactly one common element among every pair of distinct sets has cardinality at most $n$. The first proof of the general form of the Fisher's Inequality was given by K. N. Majumdar \cite{majumdar1953} using linear algebraic methods. L{\'a}szl{\'o} Babai in \cite{babai1987nonuniform} remarked that it would be challenging to obtain a proof of Fisher's Inequality that does not rely on tools from linear algebra. 
D. R. Woodall \cite{WOODALL1997} took up the challenge and gave the first fully combinatorial proof of the inequality. Below, we give a simple, alternate proof of the inequality that does not rely on tools from linear algebra.  

\begin{theorem}(Fisher's Inequality)\label{thm:1} 
Let $k$ be a positive integer and let $\mathcal{A} = \{A_1, \ldots , A_m\}$ be a family of subsets of $U = \{e_1, \ldots , e_n\}$. If $|A_i \cap A_j|=k$ for each $1 \leq i < j \leq m$, then $m \leq n$.
\end{theorem}
\begin{proof}
It is safe to assume that all the sets in $\mathcal{A}$ are of size more than $k$. 
(Otherwise, let $A \in \mathcal{A}$ be a set of size exactly $k$. Then, the set $\{B \setminus A|B \in \mathcal{A} \setminus \{A\}\}$ partitions the elements of $[n]$ not present in $A$: this leads to $m \leq n-k+1$.)
For the sake of contradiction, assume that $m \geq n+1$. 
Let $x_{i,j}$, $1 \leq i \leq m$, $1 \leq j \leq n$, be $mn$ variables with
\begin{align*}
x_{i,j}= \begin{cases}
1,\text{ if }j \in A_i \\
0,\text{ otherwise.}
\end{cases}
\end{align*}
Let $s > m^n$ be an integer. Consider a function $f:[m] \rightarrow [s]$. Let 
\begin{align}\label{eq:1}
	f(1)x_{1,1}+f(2)x_{2,1}+ \cdots +f(m)x_{m,1}=&c_1 &\hfill \text{(corresponding to element $e_1$)} \nonumber\\
	\vdots& &\\
		f(1)x_{1,n}+f(2)x_{2,n}+ \cdots +f(m)x_{m,n}=&c_n &\hfill \text{(corresponding to element $e_n$)} \nonumber
\end{align}

We define a \emph{profile} of the function $f$ corresponding to the family $\mathcal{A}$ as the $n$-tuple $(c_1,c_2,\ldots, c_n)$. Note that the number of distinct functions from $[m]$ to $[s]$ is $s^m$ and the number of distinct profiles is at most $(ms)^n$. 
Since the number of profiles is strictly less than the total number of functions from $[m]$ to $[s]$, by pigeonhole principle, it follows that there are two distinct functions $f_1, f_2$ that yield the same profile. 
Let $\tau=f_1-f_2$. Since $f_1$ and $f_2$ are distinct, $\tau$ is not the zero function. From the set of Equations (\ref{eq:1}), it follows that
\begin{align*}
\tau(1)x_{1,1}+\tau(2)x_{2,1}+ \cdots +\tau(m)x_{m,1}=&0& \hfill \text{(Equation $(b_1)$)} \nonumber\\
\vdots& &\\
\tau(1)x_{1,n}+\tau(2)x_{2,n}+ \cdots +\tau(m)x_{m,n}=&0 & \hfill \text{(Equation $(b_n)$)} \nonumber
\end{align*}
Adding the LHS and RHS of Equations $(b_1)$ to $(b_n)$, we get 
\begin{align}\label{eq:3}
\tau(1)|A_1|+\tau_2|A_2|+ \cdots+\tau(m)|A_m|=0.
\end{align}

Let $A_1=\{e_{i_1},e_{i_2},\ldots,e_{i_r}\}$.
Adding the LHS and RHS of the Equations $(b_{i_1}), \ldots, (b_{i_r})$, we get
\begin{align}\label{eq:4}
& \tau(1)|A_1|+\tau(2)|A_1 \cap A_2|+ \cdots+ \tau(m) |A_1 \cap A_m|  = 0 \nonumber \\
\implies &\tau(1)|A_1|+k(\tau(2)+\cdots+\tau(m)) = 0
\end{align}   
Writing similar equations corresponding to each set $A_i$ in $\mathcal{A}$, we get $m$ equations as follows.
\begin{align}\label{eq:5}
&\tau(1)|A_1|+k(\tau(2)+\cdots+\tau(m)) =0  \nonumber \\
&\tau(2)|A_2|+k(\tau(1)+ \tau(3)+\cdots+\tau(m)) =0  \nonumber \\
& \vdots  \\
&\tau(m)|A_m|+k(\tau(1)+ \cdots + \tau(m-1)) =0 \nonumber
\end{align}
Adding the LHS and RHS of every equation in (\ref{eq:5}), we get
\begin{align}\label{eq:6}
&\tau(1)|A_1|+\tau_2|A_2|+ \cdots+\tau(m)|A_m| + k (m-1)(\tau(1)+\cdots + \tau(m)) = 0 \nonumber\\
\implies& \tau(1)+\cdots + \tau(m) = 0  \hfill\text{ (Using Equation \ref{eq:3})}.
\end{align}

Since $\tau$ is not the zero function, without loss of generality, assume that $\tau(1) \neq 0$.
From Equation \ref{eq:4}, it follows that
\begin{align}
&\tau(1)|A_1|+k(\tau(2)+\cdots+\tau(m)) =0 \nonumber\\
\implies& \tau(1)|A_1| + k(-\tau(1)) =0 \text{ (From Equation \ref{eq:6})} \nonumber\\
\implies & \tau(1) (|A_1|-k) = 0.
\end{align}
This is a contradiction as $|A_1| > k$ and $\tau(1) \neq 0$. So, our assumption that $m\geq n+1$ is false. 
\end{proof}
\section{Concluding remarks}
The pigeonholing argument used to show that there exists a non-trivial solution to the homogeneous system of linear equations $(b_1)$ to $(b_n)$ whose coefficients are either $0$ or $1$ can be extended to any homogeneous system of $n$ linear equations on $m$ ($>n$) variables whose coefficients are integers by taking an appropriately large $s$ (Siegel's Lemma \cite{siegel1929einige}). Hence, a similar  pigeonholing argument can be used to give a proof, that does not rely on `tricks' of linear algebra, of other theorems in combinatorics that use a homogeneous system of linear equations like the Beck-Fiala Theorem \cite{beck1981integer}, Beck-Spencer Theorem \cite{beck1983balancing}, etc. In \cite{Vishwanathan13}, a counting based proof of the Graham-Pollak Theorem is given using similar ideas.  
\bibliographystyle{apalike}

\end{document}